\documentclass[11pt]{amsart} 

\issueinfo{00}{0}{Month}{Year}
\copyrightinfo{2008}{University of Calgary}
\pagespan{1}{2}
\PII{ISSN 1715-0868}

\usepackage[latin1]{inputenc}
\usepackage{subfigure}
\usepackage{calc}

\usepackage{graphicx}
\long\def\beginpgfgraphicnamed#1#2\endpgfgraphicnamed{\includegraphics{#1}}

\theoremstyle{plain} 
\newtheorem{theo}{Theorem}
\newtheorem{lemm}{Lemma} 
\newtheorem{prop}{Proposition}
\newtheorem{coro}{Corollary}

\newcommand{\n}{\mathbb{N}}
\renewcommand{\r}{\mathbb{R}}

\newenvironment{myitemize}%
{\begin{list}{\textendash}{%
      \settowidth\labelwidth{\textendash}%
      \setlength\itemindent{-\leftmargin+\labelwidth+\labelsep}}}%
{\end{list}}

\begin{document}

\title{Strictly chained $(p,q)$-ary partitions} 

\author{Laurent Imbert}
\address{CNRS, LIRMM, Université Montpellier 2\\
  161 rue Ada\\
  F-34095 Montpellier, cedex\\
  France}
\email{Laurent.Imbert@lirmm.fr}

\author{Fabrice Philippe}
\address{CNRS, LIRMM, Université Montpellier 2, Université Montpellier 3\\
  161 rue Ada\\
  F-34095 Montpellier cedex\\
  France} 
\email{Fabrice.Philippe@lirmm.fr}

\date{(date1), and in revised form (date2).}
\subjclass[2000]{05A17}
\keywords{Integer partitions, chain partitions, double-base chains}

\thanks{This work was done during the first author's leave at the University of
  Calgary's Centre for Information Security and Cryptography (CISaC), under a
  PIMS-CNRS agreement.}

\begin{abstract}
  We consider a special type of integer partitions in which the parts of the
  form $p^aq^b$, for some relatively prime integers $p$ and $q$, are restricted
  by divisibility conditions. We investigate the problems of generating and
  encoding those partitions and give some estimates for several partition
  functions.
\end{abstract} 

\maketitle

\section{Introduction} 

A partition of an integer $n$ is a non-increasing sequence of positive
integers, called parts, summing up to $n$, possibly subject to one or more
constraints. For instance, one may want the parts to be distinct, odd, prime,
powers of some integer, etc.  The most famous reference for integer partitions
is doubtless the textbook by Andrews~\cite{Andrews84}.

Strictly chained partitions are finite sequences of integers that decrease for
the divisibility order. In other words, partitions of the form $n = a_1 + a_2 +
\dots + a_k$ into distinct positive integers $a_1, \dots, a_k$ such that $a_k |
a_{k-1} | \dots | a_1$.  Binary and more generally $m$-ary partitions with
distinct parts are obviously special instances of this type of partitions.  In
general, an integer admits several such partitions; for example
\begin{align*}
  873 &= 512 + 256 + 64 +32 + 8 + 1, \\
  873 &= 720 + 120 + 24 + 6 + 2 + 1, \\
  873 &= 696+ 174 + 3.
\end{align*}
These unconventional partitions have been considered by Erd\"os and Loxton
in~\cite{ErdLox79a}. If $p(n)$ denotes the number of strictly chained partitions
of $n$ and $p_1(n)$ the number of partitions of this type whose smallest part is
1, they show that $p(n) \geq \log_2{n}$ for $n \geq 6$ and $p_1(n) \geq
\frac{1}{2}\log_2{n}$ for $n \geq 27$ except when $n-1$ is prime, in which case
$p_1(n) = 1$.  For $x$ sufficiently large, they also prove that the sum function
$ P(x) = \sum_{1 \leq n \leq x} p(n) $ behaves like $c x^\rho$, where $c$ is an
unknown constant and $\rho \approx 1.72865$ is the unique root of $\zeta(\rho) =
2$, where $\zeta$ is the Riemann zeta function.

In this paper, we investigate a special case of the type of partitions defined
above. We consider strictly chained $(p,q)$-ary partitions, i.e., partitions
with distinct parts of the form $p^aq^b$, with the further constraint that each
part is a multiple of the following one. For sake of simplicity, we further
assume that $p$ and $q$ are relatively prime integers greater than 1.

This work arose from recent developments on so-called double-base chains, in
particular their use in speeding-up exponentiation and elliptic curve scalar
multiplication~\cite{DimImbMis08:mathcomp}.  The first aim of this paper is to
provide some theoretical results as a basis of further algorithmic studies. We
investigate several problems including those of generating, encoding and
counting those partitions. We also give some results on the length (the number
of parts) of the shortest partitions of that type.  The special case
$\min(p,q)=2$ will be given a special attention, in particular the case
$(p,q)=(2,3)$. In addition, many of our results generalize almost immediately to
any number of relatively prime integers instead of $(p,q)$ only.

In the following, the set of all strictly chained $(p,q)$-ary partitions whose
parts sum up to $U$ is denoted by $\Omega(U)$, and its subset of partitions of
$U$ with no part 1 by $\Omega^{*}(U)$. Cardinalities of these sets are denoted
by $W$ and $W^{*}$ respectively.

\section{Generating partitions in $\Omega(U)$}
\label{sec:gen_partitions}

\subsection{Complete generation}
\label{sec:complete_gen}

We define three mappings from subsets $\Omega$ of $\Omega(U)$ to the set $P$ of
all (general) partitions; i.e., the non-increasing sequences of positive
integers whose sum is finite. Let $\varpi \in \Omega$.  The first map consists
in multiplying each part of $\varpi$ by $p$, and the second one does the same
with $q$ instead of $p$.  The resulting sets of partitions are denoted
$^p\Omega$ and $^q\Omega$ respectively.

Definition of the third mapping depends on $(p,q)$. If $\min(p,q)=2$, we define the
\emph{binary amount} of a partition as the sum of all its binary parts or $0$ if
none.  The third mapping is then defined as follows:
\begin{myitemize}
\item if $\min(p,q)=2$, increase by 1 the binary amount of $\varpi$,
\item if $\min(p,q)>2$, add a part one to $\varpi$.  
\end{myitemize}
In both cases, the resulting set of partitions is denoted by $^1\Omega$.  
Each of these three mapping is clearly injective. Moreover,
$^p\Omega(U) \subset \Omega(pU)$ and $^q\Omega(U) \subset \Omega(qU)$. However,
$^1\Omega(U) \not \subset \Omega(U+1)$ in general.  Indeed, if $\min(p,q)>2$,
the part $1$ may appear twice in $^1\Omega(U)$ (remember that our mappings take
their values in $P$), and even in the binary case the divisibility condition may
be lost. For example, the strictly chained $(2,3)$-ary partition $(6,2,1)$ is
turned into $(6,4)$ which is not in $\Omega(10)$.

If $U$ is a positive integer the set $\Omega(U)$ is possibly not empty (of course, if $\min(p,q)=2$
it contains at least the binary partition of $U$), whereas it is always empty if $U \not \in
\n$. We shall consider that $\Omega(0)$ contains the empty sequence $()$. This convention is
consistent with the simple but essential identities given in Lemma~\ref{lemomega} below. The first
one is readily obtained by considering partitions without or with part 1, and the second one by
noticing that, in a partition with no part 1, either $p$ or $q$ (or both) must divide the smallest
part, and then all parts. Disjoint union of sets is denoted additively.

\begin{lemm} 
  \label{lemomega}
  For any integer $U$, 
  \begin{align*}
    \Omega(U)=\Omega^*(U)+{^1\Omega^*(U-1)},\qquad
    \Omega^*(U)={^p}\Omega(U/p)\cup{^q}\Omega(U/q)
  \end{align*}
\end{lemm}
Using Lemma~\ref{lemomega}, ground value $\Omega(1)=\{(1)\}$ and convention
$\Omega(0)=\{()\}$, partitions in $\Omega(U)$ can be generated recursively.
\begin{coro}
  \label{coromega}
  Let $k_0=p^{-1}\bmod{q}$ and $\ell_0=q^{-1}\bmod{p}$. Then, for all $U \in \n$, 
  \begin{align}
    \label{omegapqU}
    \Omega(pqU)&={^p}\Omega(qU)+{^q}\big(\Omega(pU)\setminus{^p}\Omega(U)\big),\\
    \label{omegapqU+1}
    \Omega(pqU+1)&={^{1p}}\Omega(qU)+{^{1q}}\big(\Omega(pU)\setminus{^p}\Omega(U)\big),
  \end{align}
  and if $1<r<pq$  
  \begin{equation}
    \label{omegapqU+r}
    \Omega(pqU+r) =
    \begin{cases}
      {^p}\Omega(qU+k_0)+{^{1q}}\Omega(pU+p-\ell_0) & \text{if } r = k_0p\\
      {^q}\Omega(pU+\ell_0)+{^{1p}}\Omega(qU+q-k_0) & \text{if } r = \ell_0q\\
      {^p}\Omega(qU+k) & \text{if } r = kp,\,k\neq k_0\\
      {^{1p}}\Omega(qU+k) & \text{if } r = kp+1,\, k \neq q-k_0\\
      {^q}\Omega(pU+\ell) & \text{if } r = \ell q,\, \ell \neq \ell_0\\
      {^{1q}}\Omega(pU+\ell) & \text{if }r = \ell q+1,\, \ell \neq p-\ell_0\\
      \emptyset & \text{otherwise}.
    \end{cases}
  \end{equation}
\end{coro}
\begin{proof} Relations~\eqref{omegapqU} and~\eqref{omegapqU+1} are immediate.  Since
$p\wedge q=1$, for $1\leq k<q$ and $1\leq \ell<p$ we have
$\Omega^*(kp)={^p}\Omega(k)$ and $\Omega^*(\ell q)={^q}\Omega(\ell)$. Moreover,
there exist $k,\ell$ such that $r=kp$ and $r-1 = \ell q$ if and only if $k=k_0$
and $\ell = p - \ell_0$. Indeed, $(k_0, p - \ell_0)$ is the unique positive
solution of the equation $kp - \ell q = 1$. Thus the first case
in~\eqref{omegapqU+r} is proved and the second one is obtained by symmetry. The
other cases follow easily.  \end{proof}

These relations take the simplest form when $p=2$, as the last three cases
in~\eqref{omegapqU+r} disappear. Relation \eqref{omegapqU} is the less efficient
one in general because of set difference. It may be improved in this case.
\begin{prop}
  \label{propomegap2} 
  If $p=2$, we have for all $U \in \n$  
  \begin{align}\nonumber
    \Omega(qU)&={^q}\Omega(U)+{^1}\Omega(qU-1). 
  \end{align}
\end{prop}
\begin{proof}
Since $q\wedge 2=1$, $2^n \not\equiv 0 \pmod{q}$. Thus the binary amount of a
partition in $\Omega(qU-1)$ is never of the form $2^n-1$, so that adding 1 to it
does not affect the chain condition. Therefore,
${^1}\Omega(qU-1)\subset\Omega(qU)$. Moreover, the binary amount of the result
is not 0, thus
${^1}\Omega(qU-1)\subset\Omega(qU)\setminus{^q}\Omega(U)$. Finally, this
inclusion is an equality since the binary amount of every partition in
$\Omega(qU)\setminus{^q}\Omega(U)$ is positive.
\end{proof} 


\subsection{Encoding with words}
\label{sec:encoding}

When $p=2$, relations given in Corollary~\ref{coromega} provide us with a compact
representation of $\Omega(U)$. An edge-labeled weakly binary tree with root $U$ and labels
in $\{1,2,q\}$ can be used to represent the partitions in $\Omega(U)$ by words on the
alphabet $\{1,2,q\}$. (The nodes are the successive arguments of $\Omega$.)  Using this
representation, the value of $U$ is easily recovered from its encoding word: we start from
the leaves (whose value equals $1$) and go through the root $U$ while performing
operations $+1$, $\times 2$ and $\times q$ according to the edge labels.  As an example,
the tree representing $\Omega(19)$ in the case $(p,q)=(2,3)$ is given in
Figure~\ref{fig:figtree19}.  Note that the language obtained this way to represent
$\Omega(U)$ is a hypercode (no word is a sub-word of another), so that other compact
representations are available (the tree in Fig.~\ref{fig:figtree19} represents
$\Omega(19)$ as a prefix code).

\begin{figure}[h]
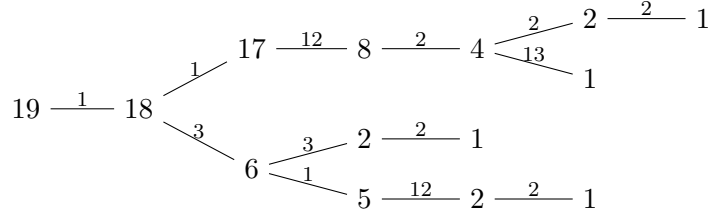

  \centering
  \beginpgfgraphicnamed{pgfgraphic_figtree19}
  \endpgfgraphicnamed
  \caption{$\Omega(19) = \{1112222, 1112213, 1332, 131122\}$ for
    $(p,q) = (2,3)$}
  \label{fig:figtree19}
\end{figure}

In the general case, relations~\eqref{omegapqU} and~(\ref{omegapqU+1}) prevent such simple, compact encodings. Nevertheless, a natural graphic
representation easily follows from the definition by noticing that strictly
chained $(p,q)$-ary partitions are partitions with distinct parts in the set
$\{p^{a}q^b,\;(a,b) \in \n^2\}$, under the constraint that couples $(a,b)$ of
exponents form a chain in $\n^2$ endowed with the usual product order.  These
chains can be encoded with words in $\{0,1,2,3\}^*$.  Consider a length-maximal
increasing path $\mathcal{P}$ in $\n^2$ containing a given chain $C$, beginning
at point (0,0) and ending at the maximal point in $C$. Such a path is called
\emph{$C$-filling}. Then, starting from the empty word, iteratively form a word
by using the following rule when progressing on $\mathcal{P}$: at point $(a,b)$,
the added letter is
\begin{center}
  \begin{tabular}{r@{\quad if }l}
    0 &  $(a,b)\not\in C$ and $(a+1,b)\in \mathcal{P}$,\\
    2 &  $(a,b)\not\in C$ and $(a,b+1)\in \mathcal{P}$,\\
    1 &  $(a,b)\in C$ and $(a+1,b)\in \mathcal{P}$,\\
    3 &  $(a,b)\in C$ and $(a,b+1)\in \mathcal{P}$ or $(a,b)$ maximal in $C$. 
  \end{tabular}
\end{center}
Given a chain $C$ in $\n^2$, there are generally many $C$-filling paths.  For
example, the chain $C = (0,0),(1,1)$ admits two $C$-filling paths, encoded by
the words 123 and 303. In order to get an unambiguous correspondence between
chains and words, we have to select a particular $C$-filling path.  We choose
the one with minimal $a$-values: for each $i$, the $i^\mathrm{th}$ point
$(a_{i},i+1-a_{i})$ in this path is such that $a_{i}$ is the minimal abscissa of
the $i^\mathrm{th}$ points of all $C$-filling paths.  In other words, we always
go North before going East as illustrated in Figure~\ref{figlattice}.  We denote
by $w_{C}$ the associated word.  When $C$ runs through all finite increasing
chains in $\n^2$, it can be shown that the words $w_{C}$ are precisely those
ending by 3 and with neither 02 nor 12 as factors. Note that the language
representing $\Omega(U)$ is clearly an infix code (no word is a factor of
another).
\begin{figure}[h]
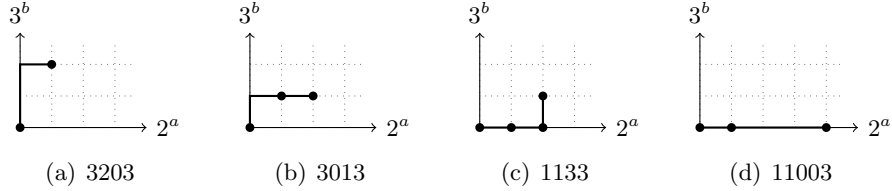

  \centering
  \subfigure[3203]{
    \beginpgfgraphicnamed{pgfgraphic_3203}
    \endpgfgraphicnamed
    \label{fig:3203}
  }
  \subfigure[3013]{
    \beginpgfgraphicnamed{pgfgraphic_3013}
    \endpgfgraphicnamed
    \label{fig:3013}
  }
  \subfigure[1133]{
    \beginpgfgraphicnamed{pgfgraphic_1133}
    \endpgfgraphicnamed
    \label{fig:1133}}
  \subfigure[11003]{
    \beginpgfgraphicnamed{pgfgraphic_11003}
    \endpgfgraphicnamed
    \label{fig:11003}}
  \caption{$\Omega(19)$ for $(p,q)=(2,3)$ and encoding words}
  \label{figlattice}
\end{figure}

\subsection{Random generation} 
\label{sec:random_gen}

The tree-based representation of $\Omega(U)$ given in the previous subsection
allows for a straightforward uniform unbiased sampling method, as soon as the
weight of each subtree is known. We show how to efficiently compute these
weights in the next section. Accordingly, defining a Markov chain is seemingly
superfluous for uniform random generation.

Nevertheless, availability of a symmetric, connected transition graph on
$\Omega(U)$ may be useful for algorithmic purposes, in particular when searching
to minimize the number of parts through recoding or while performing addition. We
construct such a graph in the special case $(p,q)=(2,3)$. Transitions are based
upon elementary identities which are represented in Figure~\ref{fig:trans}
below. The second identity is a generalization of $4 = 3+1$.

\begin{figure}[h]
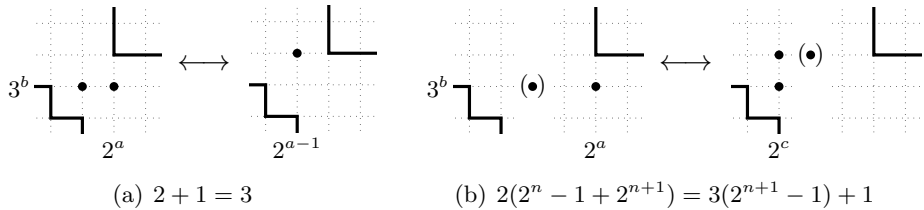

  \centering
  \subfigure[$2+1=3$]{
    \beginpgfgraphicnamed{pgfgraphic_transleft1}
    \endpgfgraphicnamed
    \raisebox{13mm}{\;\;$\longleftrightarrow$}
    \beginpgfgraphicnamed{pgfgraphic_transleft2}
    \endpgfgraphicnamed
    \label{fig:transleft}
  }
 \hfill
  \subfigure[$2(2^{n}-1+2^{n+1}) = 3(2^{n+1}-1)+1$]{
    \beginpgfgraphicnamed{pgfgraphic_transright1}
    \endpgfgraphicnamed
    \raisebox{13mm}{\;\;$\longleftrightarrow$}
    \beginpgfgraphicnamed{pgfgraphic_transright2}
    \endpgfgraphicnamed
    \label{fig:transright}
  }
  \caption{Valid symmetric transitions between partitions in $\Omega(U)$ for
    $(p,q) = (2,3)$}
  \label{fig:trans}
\end{figure}

More formally, choose a partition $\varpi$ in $\Omega(U)$. 
For some fixed value $b$, consider the part $2^a3^b$ in $\varpi$ with maximal $a$-value.  If the
part $2^{a-1}3^b$ also exists in $\varpi$ (Figure~\ref{fig:transleft}), subtracting both parts and
adding the part $2^{a-1}3^{b+1}$ results in an other partition in $\Omega(U)$. Indeed, the other
parts $2^c3^d$ in $\varpi$ satisfy either $c\geq a,\,d>b$ or $c\leq a-1,\, d\leq b$, so that the
chain constraint is respected. The bounds for authorized $(c,d)$-values are represented with bold
lines in the lower-left and upper-right corners in Figure~\ref{fig:trans}.  If the part $2^{a-1}3^b$
does not exist in $\varpi$, let $C$ be the maximal set of contiguous parts of the form $2^{a-i}3^b$,
with $i$ running from 2 to an eventual maximal value $n$ (Figure~\ref{fig:transright}).  Next
consider the value $2^{c}3^b$ with $c=a-2$ if $C$ is empty or $c=a-n-1$ otherwise. This value
is not a part by definition. If $c\geq 0$ and if there is no part in $\varpi$ of the form
$2^{c+1}3^d$ with $d<b$, we obtain another partition in $\Omega(U)$ by multiplying every part in $C$
by $3$, subtracting $2^a3^b$ from $\varpi$, and adding the two parts $2^c3^b$, $2^c3^{b+1}$.  It is
easy to see that the two families of transitions defined above are reversible, and the proof is
omitted.

Let $G(U)$ be the symmetric graph on $\Omega(U)$ corresponding to the above
transitions. The transition graph on $\Omega(27) =$ \{11013, 13003, 1333, 21003,
2133, 2213, 2223\} is given as an example in
Figure~\ref{fig:G27}.

\begin{figure}
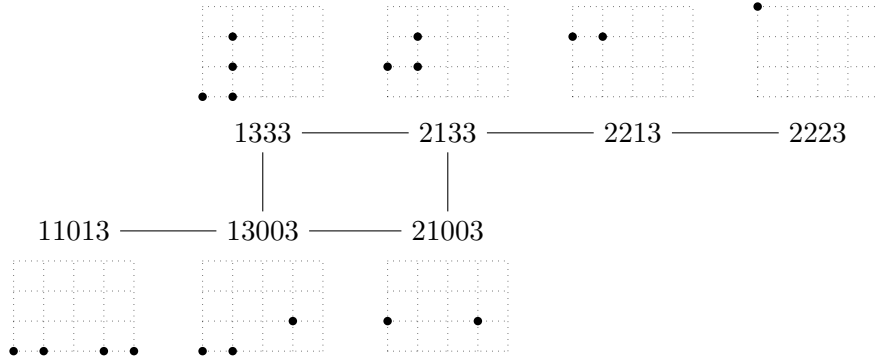

  \centering
  \beginpgfgraphicnamed{pgfgraphic_G27}
  \endpgfgraphicnamed
  \caption{The transition graph $G(27)$ for $(p,q) = (2,3)$}
  \label{fig:G27}
\end{figure}

\begin{theo}
  For $(p,q)=(2,3)$, the graph $G(U)$ is connected, and its diameter is at most
  $\dfrac{\log^2{U}}{\log{2}\log{3}}$.
\end{theo}

\begin{proof} It suffices to show that there is a path between any partition in
$\Omega(U)$ and the binary partition of $U$. Let $\varpi$ be a partition in
$\Omega(U)$ different from the binary one, let $b$ be the largest power of 3
amongst all parts in $\varpi$, and let $a$ be such that $2^a3^b$ is the smallest
part amongst those; i.e. $a = \min\{i\;; 2^i3^b \text{ a part in } \varpi\}$.
Since there is no part in $\varpi$ of the form $2^i3^d$ with $d>b$ ($b$ is
maximal), both downwards transitions are allowed.  If $2^a3^{b-1}$ is not in
$\varpi$, we apply transition $3 = 1+2$ as illustrated below.
\begin{center}
  \beginpgfgraphicnamed{pgfgraphic_3is1p2a}
  \endpgfgraphicnamed
  \raisebox{10mm}{\quad$\longrightarrow$\;}
  \beginpgfgraphicnamed{pgfgraphic_3is1p2b}
  \endpgfgraphicnamed
\end{center}
Otherwise, transition $1+3 = 4$ (or its generalization) is allowed, as shown on the
following figures.
\begin{center}
  \beginpgfgraphicnamed{pgfgraphic_1p3is4a}
  \endpgfgraphicnamed
  \raisebox{10mm}{\quad$\longrightarrow$\;}
  \beginpgfgraphicnamed{pgfgraphic_1p3is4b}
  \endpgfgraphicnamed
\end{center}
In both cases, the number of parts with power of 3 equal to $b$ is reduced by at
least one unit.  Iterating the process a finite number of times results in a
partition in $\Omega(U)$ where the largest power of 3 is less than $b$, and
further iterations eventually drop this exponent to 0; meaning that we have
reached the binary partition.

Regarding the diameter of the graph, it suffices to note that the exponents $a,b
\geq 0$ of the parts of any partition in $\Omega(U)$ are located below the line
$a\log 2+b\log 3=\log U$, so that the number of transitions between $\varpi$ and
the binary partition in the above algorithm cannot exceed the number of possible
$(a,b)$-values (with $a, b \geq 0$), that is, $\frac{\log^2{U}}{2 \log{2}\log{3}}$.
\end{proof}


\section{Shortest partitions}
\label{sec:short_part}

Finding a partition with a small number of parts, in a reasonable amount of
time, is of major importance for the applications mentioned in the introduction.
For example, in the domain of elliptic curve cryptography, the complexity of
  a scalar multiplication algorithm (the elliptic curve equivalent of an
  exponentiation) based on double-base chains~\cite{DimImbMis08:mathcomp}
  heavily depends on the number of parts in a $(2,3)$-ary partition of the
  exponent. 

If $|\varpi|$ denote the number of parts of a partition $\varpi$, we define
$\sigma(U) = \min_{\varpi \in \Omega(U)} |\varpi|$, the length of the
  shortest partitions in $\Omega(U)$.  For example, from
figure~\ref{figlattice} it is easy to see that $\sigma(19) = 2$.  The values of
$\sigma$ can be easily computed using Corollary~\ref{coromega} by noticing that
$$ 
\sigma(pqU) = \min (\sigma(qU), \sigma(pU)), \qquad
\sigma(pqU+1) = 1 + \sigma(pqU).
$$
The relations in~\eqref{omegapqU+r} can be adapted as well for numbers of the
form $pqU+r$ for $1 < r < pq$.
For $(p,q) = (2,3)$ the following Maple code can be used to
compute the first $500000$ values of $\sigma$ in approximately 1 second.

\small
\begin{verbatim}
s := proc(U)
option remember;  local r;
if U <= 2 then 1 else
   r := irem(U,6);
   if r=0 then min(s(U/3), s(U/2))
   elif r=1 then 1 + s(U-1)
   elif r=2 then s(U/2)
   elif r=3 then min(s(U/3), 1+s((U-1)/2))
   elif r=4 then min(s(U/2), 1+s((U-1)/3))
   elif r=5 then 1 + s((U-1)/2) 
   fi
fi
end:
\end{verbatim}
\normalsize As we shall see in Section~\ref{sec:case_p2}, if $(p,q)=(2,3)$ there
is only one partition for numbers of the form $2^a3-1$, so that
$\sigma(2^a3-1)=a+1$, whereas $\sigma(2^a3)=1$. Our numerical experiments
suggest that $\sigma(U) \approx (\log_2{U})/4$ on average, which empirically
confirms an intuition that $\sigma(U)$ is unfortunately not in $o(\log{U})$,
even in a loose sense.
 
It is perhaps possible to significantly reduce the minimum number of parts by
considering \emph{signed} strictly chained partitions. For instance, if one
allows the largest part in any $\varpi \in \Omega(U)$ to be less than or equal
to $U+1$, while allowing the other parts to be either added or subtracted, what
is the length of the shortest such partition(s) of $U$? Studying similar relaxed
definitions of chained partitions for the same purpose shall be the object of further research.


\section{Computing $W(U)$}
\label{sec:computing_W}

In this section we provide different formulas to compute $W(U)$, the cardinality of
$\Omega(U)$, with a special attention given to the case $p=2$.

\subsection{Simple identities}
\label{sec:simple_id}

From Lemma~\ref{lemomega} it follows immediately that, for all $U \in \n$,
we have
\begin{align}
  W(U)   &= W^*(U) + W^*(U-1), \\
  W^*(U) &= W(U/p) + W(U/q) - W(U/pq). 
\end{align}
Corollary~\ref{coromega} and Proposition~\ref{propomegap2} are easily translated
in the same way.  Simply note that the sets involved in the first two cases
of~\eqref{omegapqU+r} are disjoint.  Let $k_0=p^{-1} \bmod{q}$ and
$\ell_0=q^{-1} \bmod{p}$.  Then, for all $U \in \n$, we have
\begin{equation}
  \label{WpqU}
  W(pqU) = W(pqU+1) = W(pU) + W(qU) - W(U),
\end{equation}
and, for $1 < r < pq$,
\begin{equation}
  \label{WpqU+r}
  W(pqU+r) =
  \begin{cases}
    W(qU+k_0)+W(pU+p-\ell_0) & \text{if } r = k_0p\\[6pt]
    W(pU+\ell_0)+W(qU+q-k_0) & \text{if } r = \ell_0q\\[6pt]
    W(qU+k) &
    \text{if } r=kp,\,k\neq k_0\\
    & \text{or } r=kp+1,\,k\neq q-k_0\\[6pt]
    W(pU+\ell) &
    \text{if } r=\ell q,\,\ell\neq \ell_0 \\
    &\text{or } r=\ell q+1,\,\ell\neq p-\ell_0\\[6pt]
    0 & \text{otherwise}.
  \end{cases}
\end{equation}

The case $p=2$ allows for further simplifications. 

\begin{prop} If $p=2$, then for all $U \in \n$ we have
\begin{align}
  \label{WqU}
  W(qU) &= W(U) + W(qU-1),\\[0.5em]
  \label{WqU+1}
  W(qU+1) &= \left\{
    \begin{array}{ll}
      W(U) + W \left(\dfrac{qU}{2}-1 \right) & \text{for even }U,\\[1em]
      W(U) + W \left(\dfrac{qU+1}{2} \right) & \text{for odd }U,
    \end{array}
  \right.\\[0.5em]
  \label{WqU+r}
  W(qU+r) &= W\left(\left\lfloor\frac{qU+r}{2}\right\rfloor\right) \qquad\text{ if } 2
  \leq r \leq q-1.
\end{align}
\end{prop}

\begin{proof}
Relation~\eqref{WqU} follows from Prop.~\ref{propomegap2} at once.
 Now, since $p=2$ we have $\ell_0=1$ and $k_0= \frac{q+1}{2}$. Then, for $1<r<2q$,
 relations given in~\eqref{WpqU+r} read 
\begin{align}
  \label{W2qU+r}
  W(2qU+r) &= \left\{
    \begin{array}{ll}
      W(qU+\frac{q+1}2)+W(2U+1) & \mathrm{if\ } r=q+1\\[4pt]
      W(2U+1)+W(qU+\frac{q-1}2) & \mathrm{if\ } r=q\\[4pt]
      W(qU+k) &
      \left|
        \begin{array}{l}
          \mathrm{if\ } r=2k,\,k\neq \frac{q+1}2\\[4pt]
          \mathrm{or\ }r=2k+1,\,k\neq \frac{q-1}2
        \end{array}
      \right.
    \end{array}
  \right.
\end{align}
The last relation is summarized in~\eqref{WqU+r} by considering odd and even
$U$. The second one is a particular case of~\eqref{WqU}, and the first one is
identical to \eqref{WqU+1} for odd $U$. 
Indeed, according to~\eqref{W2qU+r} we have $W(qU+\frac{q-1}2)=W(2qU+q-1)$ and
$W(qU+\frac{q+1}2)=W(2qU+q+2)$.
Finally,~\eqref{WqU+1} for even $U$ follows from~\eqref{WpqU} and~\eqref{WqU}.
\end{proof}

Note that if $(p,q)=(2,3)$ relations~\eqref{WqU+1} may be summarized as
\begin{equation}
  \label{eq:W3U1}
  W(3U+1) = W(U) + W\left(3\left\lfloor\frac{U+1}{2}\right\rfloor - 1 \right).  
\end{equation}

\subsection{A general relation}
\label{sec:gen_rel}

Let $W_{p}(U)$ be the number of partitions of $U$ with distinct parts taken in
$\{p^n,n \in \n\}$.  Clearly $W_p(U) \in \{0,1\}$, depending whether or not $U$
can be written in base $p$ with digits in $\{0,1\}$ only.  As an example,
Erd\"os and Graham conjectured~\cite{ErdGra80} that the only powers of 2
  such that $W_3(2^n) = 1$ are 1, 4 and 256. (This has been verified by
  Vardi~\cite{Vardi91} up to $2 \times 3^{20}$.) Of course, since the binary
expansion of $U$ is unique, we have $W_2(U) = 1$ for all $U > 0$.  It is easy to
see that $W_{p}(kp+1)=W_{p}(kp)=W_{p}(k)$ and $W_{p}(kp+r)=0$ if $r \bmod{p}
\not\in\{0,1\}$. Together with the initial condition $W_{p}(0)=1$, these
relations define $W_{p}$ recursively.  Those quantities are involved in the
following theorem.

\begin{theo}
  \label{thWpq}
  If $U \not \in \n$, then $W(U) = 0$. Otherwise, for $U \geq 1$, we have 
  \begin{gather}
    \label{eq:recdir}
    W(U) = W_p(U) + W\left(\frac{U}{q}\right) +
    \sum_{c=0}^{\left\lfloor\log_{p}(\frac{U}{q+1})\right\rfloor}
    \delta_{p,q}(c,U)\,W\left(\left\lfloor\frac{U}{p^{c}q}\right\rfloor\right),\\[1em]
    \text{where }
    \label{eq:deltapq}
    \delta_{p,q}(c,U) = 
    \left\{
      \begin{array}{ll}
        1 & \text{if } \left\lfloor U/p^c \right\rfloor \equiv 1\pmod{q}
        \text{\; and\; }
        W_p(U \bmod{p^c}) = 1\\[0.5em]
        0 & \text{otherwise}
      \end{array}
    \right.
  \end{gather} 
\end{theo}
\begin{proof} Let us sort the partitions in $\Omega(U)$ with respect to their $p$-ary
amount, that is, the sum of all parts of the form $p^a$.  There are $W(U/q)$
such partitions with $p$-ary amount equal to 0.  This occurs either when there
is none or when $q$ divides all parts.  Next, note that a partition $\varpi$
in $\Omega(U)$ with $p$-ary amount $n \geq 1$, assuming it exists, is perfectly
described by its non $p$-ary parts, i.e., parts of the form $p^aq^b$ with $b >
0$.  Indeed, if $c_{n}=\left\lfloor\log_{p}(n)\right\rfloor$ denotes the largest
power of $p$ among the $p$-ary parts of $\varpi$, each of the non $p$-ary parts
of $\varpi$ is clearly a multiple of $p^{c_{n}}q$, so that dividing each of them
by $p^{c_{n}}q$ yields a characterizing partition of $(U-n)/(p^{c_{n}}q)$. Since
this correspondence is clearly one-to-one and onto
$\Omega\left((U-n)/(p^{c_{n}}q)\right)$, the number of partitions in $\Omega(U)$
with $p$-ary amount equal to $n$ is exactly $W\left((U-n)/(p^{c_{n}}q)\right)$.
Finally, $n$ is a $p$-ary amount if and only if $W_p(n)=1$. Accordingly,
\begin{equation*}
  W(U) = W(U/q)+\sum_{n=1}^{U}W_{p}(n)\,W\left(\frac{U-n}{p^{c_{n}}q}\right).
\end{equation*}
Splitting the interval $[1,U]$ into sub-intervals of the form $[p^c, p^{c+1})$
for $c \geq 0$, yields
\begin{equation}
  \label{eq:rec}
  W(U) = W(U/q) + \sum_{c \geq 0} \sum_{n=p^{c}}^{p^{c+1}-1}
  W_{p}(n)\,W\left(\frac{U-n}{p^{c}q}\right).
\end{equation}
If $p^c \leq n < p^{c+1}$ then we have $U-n\geq p^cq$ only if $U \geq
(q+1)p^{c}$, so that $c$-summands vanish for $c > \log_{p}(U/(q+1))$ except for
the term $W_p(U)\,W(0)=W_p(U)$.

Finally, let $c\leq \log_{p}(U/(q+1))$. We have $n \in [p^c,p^{c+1})$ and
$W_{p}(n)=1$ if and only if $n=p^c+r$ with $0\leq r<p^c$ and $W_p(r)=1$. Thus
\begin{equation*}
  \sum_{n=p^{c}}^{p^{c+1}-1} W_{p}(n)\,W\left(\frac{U-n}{p^{c}q}\right)
  =\sum_{r=0}^{p^{c}-1} W_{p}(r)\,W\left(\frac{U-p^c-r}{p^{c}q}\right). 
\end{equation*}
Moreover, $p^cq$ divides $U-p^c-r$ with $0\leq r<p^c$ only if $U \bmod{p^cq}=p^c+r$. Therefore,
\begin{equation*}
  \sum_{r=0}^{p^{c}-1} W_{p}(r)\,W\left(\frac{U-p^c-r}{p^{c}q}\right)=
  \delta_{p,q}(c,U)\,W\left(\left\lfloor\frac{U}{p^{c}q}\right\rfloor\right), 
\end{equation*}
where $\delta_{p,q}(c,U)$ is equal to 1 if $r=U \bmod{p^cq}-p^c$ satisfies $0\leq
r<p^c$ and $W_p(r)=1$, otherwise it vanishes. Writing $U=kp^cq+p^c+r$, it is not
difficult to show that this definition of $\delta_{p,q}(c,U)$ is equivalent to
the one given in \eqref{eq:deltapq}.  \end{proof}

It is worthwhile to point out that if $W_p(U\bmod p^c)=0$ holds for a given $c$
it also holds for all $c'\geq c$. In other words, it means that the number of
$c$-summands in the r.h.s. of~\eqref{eq:recdir} is at most equal to the the
weight of the first digit (starting from the unit) greater than 1 in the
expansion of $U$ in base $p$.  Of course, the latter remark is useless if $p=2$.
Nevertheless, even in this case many summands do vanish, as shown next.

\begin{prop}
  \label{few}
  Assume $p<q$ and let $N=\lfloor\log_p(q-\frac{q-1}p)\rfloor$. \\[.5em]
  For all $U$ and $c$ in $\n$, if $\delta_{p,q}(c,U)=1$ then
  $\delta_{p,q}(c+k,U)=0$ for all integers $k$ such that $k\geq-c$ and
  $0<|k|\leq N$.
\end{prop}
\begin{proof} 
Let us write
$$
  U = (u_n, \dots, u_{c}, u_{c-1}, u_{c-2}, \dots, u_1, u_0)_p = \lfloor U/p^{c} \rfloor p^{c} + U
  \bmod{p^{c}}.
$$
If $\delta_{p,q}(c,U)=1$, Theorem~\ref{thWpq} ensures that there exists $\ell$
such that $\lfloor U/p^{c} \rfloor = \ell q + 1$. We thus have, for each $1\leq
k\leq c$,
\begin{align*}
  \lfloor U/p^{c-k} \rfloor = p^{k}(\ell q+1) + \sum_{i=1}^k p^{k-i}u_{c-i}. 
\end{align*}
Moreover, each $u_i$ with $i<c$ is either equal to 0 or to 1 since $W_p(U \bmod{p^c}) =
1$. Accordingly, $p^{k}+ \sum_{i=1}^k p^{k-i}u_{c-i}$ is an integer in $\left[
  p^k,\frac{p^{k+1}-1}{p-1} \right]$.  Therefore, if we moreover assume $k\leq N$, then it is easy
to see that $\lfloor U/p^{c-k} \rfloor$ cannot be equal to 1 modulo $q$, so that
$\delta_{p,q}(c-k,U)=0$.  Finally, suppose there exists $k\in[1,N]$ such that
$\delta_{p,q}(c+k,U)=1$. Then, according to the above discussion, $\delta_{p,q}(c,U)$ should then be
equal to 0, which contradicts our hypothesis. \end{proof}

As an example, consider $(p,q)=(2,3)$.  Here  $N=1$ thus at most one of two
consecutive summands is non-zero. An extremal case is $U=4^a$, because
$\delta_{2,3}(c,4^a)=1$ if, and only if, $c$ is even.

\section{The sequence $W$}
\label{sec:seqW}

\subsection{Generalities} 
\label{sec:generalities}

For any given pair $p,q$, the sequence $W$ behaves rather irregularly, as one is
easily convinced by computing its first values (see Fig.~\ref{fig:first400}).

\begin{figure}[htbp]
  \centering
  \includegraphics[width=\textwidth]{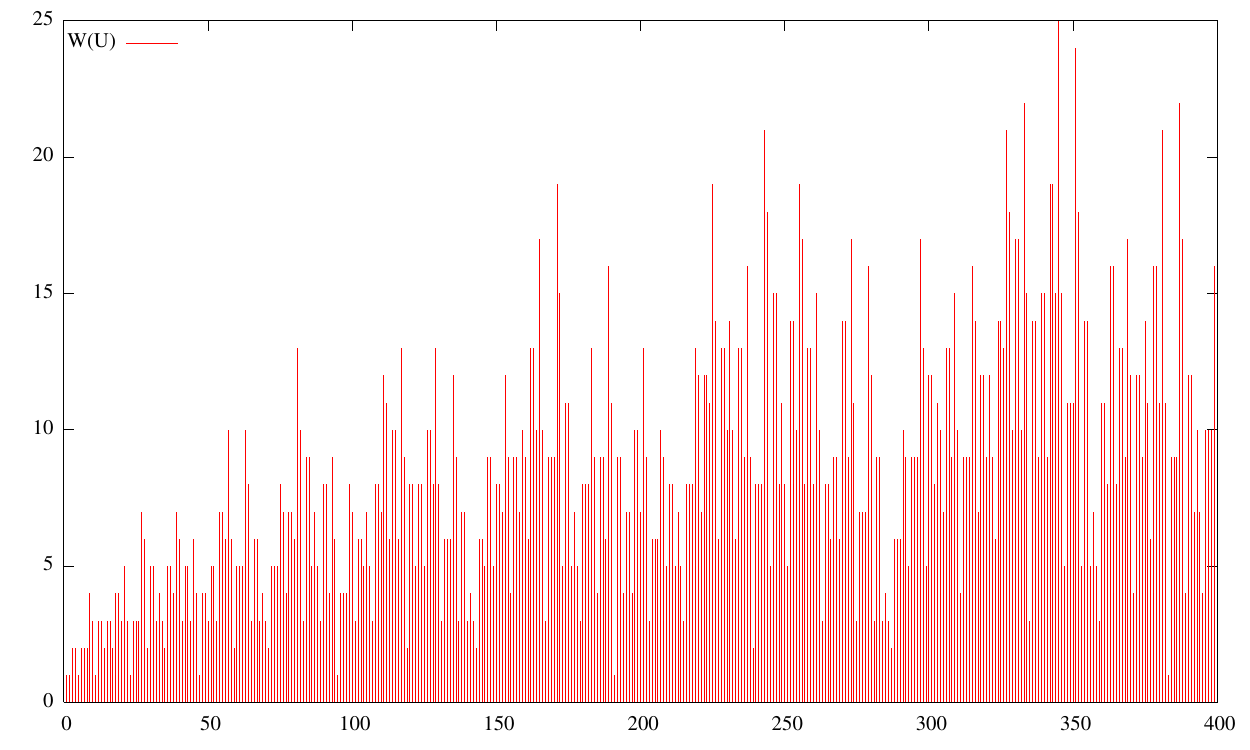}
  \caption{The first $400$ values of $W(U)$ for $(p,q) = (2,3)$}
  \label{fig:first400}
\end{figure}

On the one hand, if $\min(p,q)>2$, then $W$ takes infinitely often the value 0 
according to \eqref{WpqU+r}. In addition, it takes infinitely often the value 1 if $p=2$.  Indeed,
iterating \eqref{WqU+r} shows that, for all integers $U$ and $a$,
\begin{equation}
  \label{WqU-1red}
  W\left(2^{a}q(2U+1)-1\right)=W\left(qU+\frac{q-1}{2}\right). 
\end{equation}   
Therefore, for $p=2$, the sequence $W$ takes infinitely many times each value of
the form $W(qU+\frac{q-1}2)$, and in particular the value $1=W(\frac{q-1}2)$
since the binary partition is obviously the only $(2,q)$-ary partition for
  numbers less than $q$.  

On the other hand, we also have the following result. 
\begin{lemm}
The sequence $W$ is either $\{0,1\}$-valued or unbounded.
\end{lemm}
\begin{proof} Assume there exists $U$ such that $W(U)>1$. Choose two chained partitions
of $U$, with respective greatest parts $p^aq^b$ and $p^{a'}q^{b'}$, and let
$c=\max(a,a')$ and $d=\max(b,b')$. Next construct a sequence $(U_n)$ as follows:
$U_1=U$, and $U_{n+1}=(1+p^{nc}q^{nd})U_n$ for positive $n$. Remembering the
lattice representation of chained partitions, see Fig.~\ref{figlattice}, it is
easy to show by induction that $W(U_n)\geq 2^n$.  \end{proof} 

For the time being, we are not aware of any pair $(p,q)$ for which $W$ is
$\{0,1\}$-valued.  This cannot of course happen if $q$ writes in base $p$
with digits $\in\{0,1\}$ only, in particular if $p=2$. More generally, $W$ is
unbounded as soon as there exists $U$ that writes in both bases $p$ and $q$ with
digits $\in\{0,1\}$ only. But whether such an $U$ always exists or not is a
seemingly difficult problem, and only a positive answer would also solve ours.

Let us next give a coarse majoration of $W(U)$. 
\begin{prop}\label{majW}
Let $\beta\in(0,1)$ be the unique solution of $1/p^\beta +1/q^\beta=1$. Then
$W(U)\leq U^\beta $ for $U\geq 1$. 

\end{prop}
\begin{proof} Let us assume $p<q$ and show that, for all $n \in \n^*$, $W(U)\leq U^\beta
$ for $1\leq U\leq np$.  This is true for $n=1$ since for $1 \leq U \leq p$,
$W(U)\in\{0,1\}$ and $U^\beta \geq 1$. Suppose this is also true for $1 \leq n
\leq m$ and consider $U\in[mp+1,mp+p]$. Next write $U=pqV+r$ with $0\leq r\leq
pq$. If $r=1$ then by \eqref{WpqU} $W(U)=W(U-1)$ thus $W(U)\leq
(U-1)^\beta<U^\beta$ since $U-1=mp$. In the same way, if $r\neq 1$, first notice
that the arguments of $W$ in the r.h.s. of \eqref{WpqU} and \eqref{WpqU+r} are
either non-integers or lie in $[1,U/p]$, thus in $[1,mp]$ since $p\geq 2$ and
$m\geq 1$. It suffices then to consider the upper bounds for the terms in the
r.h.s.\ of~\eqref{WpqU} and~\eqref{WpqU+r} and to check that it also holds for
the l.h.s.\ term by definition of $\beta$. For example, if $r=0$, we get
from~\eqref{WpqU} that
$$
W(U)\leq W(U/p)+W(U/q)\leq (1/p^\beta +1/q^\beta)U^\beta=U^\beta. 
$$  
The other cases follow easily. \end{proof}

A direct consequence of the above is that for the particular case $(p,q) =
(2,3)$ we have $W(U) \leq U^{0.79}$. Note that our numerical computations suggest an
exponent close to $0.535$ for the best upper bound, whereas refining the above
method by taking into account~\eqref{WqU+r} only gives approximately $0.66$.

We further analyze and discuss the case $p=2$ in the next section.

\subsection{The case $p=2$}
\label{sec:case_p2}

Although $W$ behaves quite irregularly on the large scale, its local
variations obey some rules, with a main pattern of length $q$ when $p=2$.

\begin{theo}
  \label{th_ineqW}
  For all $U\in\n$, we have if $p=2$ 
  \begin{equation}
    \label{ineqW-1}
    W(qU) \geq W(qU+1) \geq W(qU-1). 
  \end{equation}
  Moreover, for $0 \leq r < q-1$, 
  \begin{equation}
    \label{ineqW}
    W(qU+r) \geq W(qU+r+1). 
  \end{equation}
\end{theo}

\begin{proof} For $1 \leq i \leq q-1$, the binary amount of any partition in
$\Omega(qU+i)$ is clearly not 0. It may thus be decreased by 1. This
transformation provides us with an injective mapping from $\Omega(qU+i)$ to
$\Omega(qU+i-1)$, so that \eqref{ineqW} and the first inequality
in~\eqref{ineqW-1} hold.

Next consider a partition $\varpi\in\Omega(qU-1)$ with a binary amount $\beta$
satisfying $\beta<2^N$. As already noticed, $\beta$ cannot be of the form
$2^n-1$ with $n\in\n$. If it is neither of the form $2^n-2$, it may be increased
by 2 to obtain a partition in $\Omega(qU+1)$ because $\beta+2$ is still less
than $2^N$. Note that the resulting partition has \emph{at least two} binary
parts.  Finally assume that $\beta=2^n-2$ with $n\geq 2$. Since $\varpi \in
\Omega(qU-1)$, we have $\beta \equiv -1 \pmod{q}$, and thus $n$ must satisfy $2^n
\equiv 1 \pmod{q}$. So, let $k=(2^n-1)/q$.  Since $k<2^{N-1}$, $\varpi$ may be turned
into a partition in $\Omega(qU+1)$ by writing $k$ in base 2 then using the
identity $2^n = qk + 1$. Note that the resulting partition has \emph{exactly
  one} binary part (the part 1).  Therefore, these two operations provide us
with an injective mapping from $\Omega(qU-1)$ to $\Omega(qU+1)$, which completes
the proof of the second inequality in~\eqref{ineqW-1}.  \end{proof}

Accordingly, jumps of the function $\max{W}: x \mapsto \max_{U\leq x}
W(U)$ occur for certain $x \in q\n$.  
We conjecture that such jumps only occur for values $x \in q(2\n+1)$.  For
  example, if $(p,q)=(2,3)$ the first values of $\max{W}$ are: 2, 4, 5, 7, 10,
  13, 17, 19, 21, 22, 25. They occur at $U =$ 3, 9, 21, 27, 57, 81, 165,
  171, 243, 333, 345; all of the form $6k+3$.

Nevertheless, the best we are able to prove is that, if exceptions do exist,
  they are of the form $x = 2q^2U$.
  Indeed, first note that values $x \in 2q\n$ are all of the form
  $2q(qU+r) = 2q^2U + 2qr$ for some $0 \leq r \leq q-1$.
Then, by~\eqref{WpqU} and~\eqref{WpqU+r} we have
\begin{align}
  W(2q^2U+2qr)&=W(2qU+2r)+W(q^2U+qr-1), \\
  W(2q^2U+2qr-q)&=W(2qU+2r-1)+W(q^2U+qr-\frac{q+1}{2}).
\end{align}
Using Theorem~\ref{th_ineqW}, we know that $W(q^2U+qr-1) \leq
W(q^2U+qr-\frac{q+1}{2})$ for all $r$, and also $W(2qU+2r) \leq W(2qU+2r-1)$ for
$1 \leq r \leq q-1$ since $q$ is odd. Therefore, $2q^2U+2qr$ cannot be a jump
point of $\max{W}$ for the latter values of $r$.
Hence, the only possible exception may only occur for $r=0$, i.e., for
  numbers of the form $2q^2U$.
We have computed the values of $\max{W}(U)$ for $U \leq 10^6$ and
$q=3,5,7,9,11,13,15$ and have not found any such exception.


We have seen that $W(U)$ takes infinitely often the value 1 when $p=2$.  In
fact, we can be a little more precise, using a simple method illustrated below
in the case $q=3$.

\begin{prop}
  \label{W=1} 
  Let $(p,q)=(2,3)$. 
  We have $W(U) = 1$ if and only if, either $U \in \{0,1\}$ or $U = 2^{a}3 - 1$
  for some $a \in \n$. 
  Also, we have $W(U) = 2$ if and only if, either $U \in \{3,4,6,7\}$ or $U = 2^{a}9 -
  1$ or $U = 2^{a}15 - 1$ for some $a \in \n$.
\end{prop}
\begin{proof} According to~\eqref{WqU} and~\eqref{WqU+1}, if $U \not\equiv -1 \pmod{3}$ then
$W(U) \geq 2$ unless $U$ is either 0 or 1. If $U \equiv -1 \pmod{3}$ then
by~\eqref{WqU-1red}, we have $W(U) = W(2^a3(2V+1)-1) = W(3V+1) = 1$ iff $V=0$.
The same reasoning shows that solving $W(U)=n$ for $U$ only requires solving it
for $U \equiv 0,1 \pmod{3}$. For $n=2$, then from~\eqref{WqU} and since
$W(0)=1$ we have $W(3U)=2$ iff $W(U)=W(3U-1)=1$, that is, either $U=1$ or $U$
and $3U-1$ are both of the form $2^{a}3-1$. The latter condition implies $U=2$,
so that $3U\in\{3,6\}$.  Next, from~\eqref{WqU+1}, we have $W(3U+1)=2$ iff
$W(U)=W(3\lfloor\frac{U+1}2\rfloor-1)=1$, that is, either $U\in\{1,2\}$ or $U$
and $3\lfloor\frac{U+1}2\rfloor-1$ are both of the form $2^{a}3-1$ with
$a>0$. The latter condition is impossible, thus $3U+1\in\{4,7\}$. Finally,
replacing $U \in \{1,2\}$ in~\eqref{WqU-1red}, leads to the numbers of the form
$2^a9-1$ and $2^a15-1$.  \end{proof}

Our numerical experiments suggest that, more generally, all values in $\n$ are taken
infinitely many times by $W$, but here again we did not succeed in either proving or
disproving this.
 

\section{Asymptotics}

As noticed previously, it seems rather difficult to give an estimate for the
  asymptotical behavior of $\max{W}$. Nevertheless, it is possible to precisely
describe the asymptotical behavior of the \emph{average} of $W$. In this
  regard, our results are similar to those of~\cite{ErdLox79a}, except the fact
  that we obtain a more explicit constant in
Theorem~\ref{equivS} below.

\begin{lemm}\label{lemrelsWx} 
Let $S(x) = \sum_{1\leq U\leq\lfloor x\rfloor}W(U)$. Then, for all $x \in
\r_{+}$,
we have
$$
S(x)=2 \left( S\left(\frac{x}{p}\right) + S\left(\frac{x}{q}\right) -
  S\left(\frac{x}{pq}\right) \right) + 1 - W^*\left(\left\lfloor x \right\rfloor
\right).
$$
\end{lemm}

\begin{proof} The relation holds for $x<1$ since both members are 0, with the usual
conventions that $W^*(0)=1$ and that a sum vanishes if the lower index exceeds
the upper one. Next assume $x\geq 1$. We have by Lemma~\ref{lemomega}
\begin{equation}
  \sum_{U=1}^{\lfloor x\rfloor}W(U) =\\
  \sum_{U=1}^{\lfloor x\rfloor}\left(W^*(U)+W^*(U-1)\right)
  = W^*(0)-W^*(\lfloor x\rfloor)+2\sum_{U=1}^{\lfloor x\rfloor}W^*(U)
  \end{equation}
  and
  \begin{equation}
    \sum_{U=1}^{\lfloor x\rfloor}W^*(U) = \sum_{U=1}^{\lfloor
      x\rfloor} \left(W(U/p)+W(U/q)-W(U/pq)\right),
  \end{equation}
from which the relation immediately follows since $\lfloor\lfloor
x\rfloor/n\rfloor=\lfloor x/n\rfloor$ if $n\in\n$.  \end{proof}

Accordingly, if there exists $\alpha$ such that $S(x) = O(x^\alpha)$ then $\alpha$ must satisfy 
\begin{align}\label{defalpha}
1/p^\alpha+1/q^\alpha-1/(pq)^\alpha=1/2. 
\end{align}
It is not difficult to check that \eqref{defalpha} has a unique positive
solution, which moreover satisfies
\begin{equation}
  \label{alpha1}
  \begin{cases}
    \alpha>1 & \text{if } \min(p,q)=2, \\
    \alpha =1 & \text{if } (p,q)=(3,4), \\
    \alpha<1 & \text{otherwise.}
  \end{cases}
\end{equation}
As a first consequence of Theorem~\ref{equivS} below and \eqref{alpha1}, the average value of $W$
goes to infinity if $p=2$, whereas it goes to 0 if $\min(p,q)>2$ except for (3,4) where it goes to a
constant in the order of 1.

Besides, $\alpha$ is greater than $\beta$ defined in Proposition~\ref{majW}.
Indeed, $(pq)^{-\beta}<1/4$ since $p^{-\beta}+q^{-\beta}=1$, thus
$$
1/p^\beta+1/q^ \beta-1/(pq)^ \beta >3/4. 
$$
Equation \eqref{defalpha} also reads 
$$
(1-p^{-\alpha})^{-1}(1-q^{-\alpha})^{-1}=2,
$$
which is reminiscent of the equation $\zeta(\rho)=2$ mentioned in the
introduction. Nevertheless, we do not try to adapt the arguments used in the
quoted paper. In order to cope with the recursive expression in
Lemma~\ref{lemrelsWx}, we rather exploit a method introduced ten years after by
Erd\"os \emph{et al.}~\cite{ErdHilOdlPudRez87}. Although the latter recursion
does not fit the required form because of its negative coefficient, it may be
rectified in our special case.

\begin{theo}\label{equivS}
  Let $\alpha$ be the positive solution of equation \eqref{defalpha}. We have,
  for large $x$,
  $$ S(x)= C_{p,q}x^\alpha(1+o(1)), $$
  where $C_{p,q}$ is a computable positive constant satisfying
  $$ C_{p,q} < 2\left(\frac{\ln p^\alpha}{p^\alpha -1}+\frac{\ln q^\alpha}{q^\alpha
    -1}\right)^{-1}. $$
\end{theo}
\begin{proof} 
  Let $f(x)=x^{-\alpha}S(x)$ for positive $x$. 
  By Lemma~\ref{lemrelsWx}, we have for $x>0$ and $i\in\n$
  \begin{multline}
    \label{relfi2}
    \frac{1}{p^{\alpha i}} f\left(\frac{x}{p^i}\right) = 
    \frac{2}{p^{\alpha(i+1)}} f\left(\frac{x}{p^{i+1}}\right) \\
    +\frac{2}{p^{\alpha i}q^\alpha} \left( f\left(\frac{x}{p^iq}\right) -
      \frac{1}{p^\alpha} f\left(\frac{x}{p^{i+1}q}\right) \right) \\
    +\frac{1}{x^\alpha} \left(1-W^*\left(\left\lfloor \frac{x}{p^i} \right\rfloor\right)\right).
  \end{multline}
  Since all terms vanish as soon as $p^i>x$, summing both sides for $i$ results, for all $x>0$, in 
  \begin{equation}
    \label{relfx2}
    f(x) = \frac{2}{q^\alpha} f\left(\frac{x}{q}\right) 
    + \sum_{i\geq 1} \frac{1}{p^{\alpha i}} f\left(\frac{x}{p^i}\right)
    +\frac{1}{x^\alpha} \sum_{i=0}^{\left\lfloor\log_p x\right\rfloor}
    \left(1-W^*\left(\left\lfloor \frac{x}{p^i} \right\rfloor\right)\right). 
  \end{equation}

  Now, let $\mu$ be the discrete measure with masses $2\,q^{-\alpha}$ at point
  $c=\log_p q$, and $p^{-\alpha i}$ at point $i$ for each $i\in\n^*$. Using
  \eqref{defalpha}, it is easy to check that $\mu$ is a probability measure,
  with expectation
  \begin{equation}
    \label{Emu2}
    E(\mu) = \frac{2}{q^\alpha}c + \frac{p^\alpha}{(p^\alpha-1)^2}.
  \end{equation} 
  Let $g(x)=f(p^x)$. Since $f$ vanishes on $(0,1)$, $g(x)$ vanishes for $x<0$.
  Therefore, and according to \eqref{relfx2}, $g$ satisfies on $\r_+$ the
  renewal equation
  \begin{align}
    g(x) &= \frac{2}{q^\alpha}g(x-c)
    + \sum_{i\geq 1}\frac{1}{p^{\alpha i}} g(x-i)
    + \sum_{i=0}^{\left\lfloor x \right\rfloor} \frac{1}{p^{\alpha x}}
    \left(1-W^*\left(\left\lfloor p^{x-i}\right
        \rfloor\right)\right) \label{renewal1} \\
    &= \int_0^x g(x-t)\mu(dt)+h(x), \label{renewal2} \\
    &\hspace{1in} \text{where } h(x)=\frac{1}{p^{\alpha x}} \sum_{i\leq
      x} \left(1-W^*\left(\left\lfloor p^{x-i}\right\rfloor\right)\right). \notag
\end{align}

In order to use the Key Renewal Theorem, we now check that $h$ is directly
Riemann-integrable (d.R.i.), that is, both $h^+=\max(h,0)$ and $h^-=\max(-h,0)$
are d.R.i.\ (see, e.g.,~\cite[p.154]{Asmussen03}). First notice that
$x\mapsto\sum_{i\leq x}W^*(\lfloor p^{x-i}\rfloor)$ is a constant on each
interval $[\log_pn,\log_p(n+1))$, $n\in\n^*$.  Thus $h$ , and $h^+$ and $h^-$ as
well since $p^{-\alpha x}$ decreases with $x$, are continuous a.e.\ with respect
to Lebesgue measure.  Next, since $W^*(U)\leq U^\beta$ with $\alpha>\beta>0$,
\begin{equation}
  h^-(x)\leq\sum_{i\leq x}\frac{W^*\left(\left\lfloor
        p^{x-i}\right\rfloor\right)}{p^{\alpha x}} \leq
  \sum_{i\leq x}\frac{p^{\beta(x-i)}}{p^{\alpha x}}
  = \frac{p^{\beta x}}{p^{\alpha x}}\sum_{i=0}^{|x|}p^{-\beta i}
  \leq \frac{p^{(\beta-\alpha) x}}{1-p^{-\beta}}. 
\end{equation}
Since $x\mapsto p^{(\beta-\alpha) x}(1-p^{-\beta})^{-1}$ is Lebesgue integrable
and decreasing, it is d.R.i., thus $h^-$ is d.R.i.\ too.  Finally, $h^+$ is also
d.R.i.\ since it is dominated by $x\mapsto(|x|+1)p^{-\alpha x}$, which is itself
d.R.i.\ since it is Lebesgue integrable and decreases for $x>\ln p^\alpha-1$.

According to the Key Renewal Theorem, we thus have 
$$
\lim_{x\to\infty} f(x) = \lim_{x\to\infty} g(x)
= \frac{1}{E(\mu)} \int_{\r_+} h(x)dx = C_{p,q}.
$$

Let us next evaluate the latter Riemann-integral.  
\begin{multline}
  \int_{\r_+}h(x)dx \leq
  \int_{\r_+}\left(\lfloor x\rfloor+1\right) \frac{dx}{p^{\alpha x}}\\
  = \sum_{n\geq 0}(n+1)\int_{n}^{n+1}\frac{dx}{p^{\alpha x}}
  = \frac1{(1-p^{-\alpha})\ln p^\alpha}. 
\end{multline}
Accordingly, and making use of \eqref{defalpha},
\begin{align}
  C_{p,q} &\leq \frac{p^\alpha}{(p^{\alpha}-1)\ln p^\alpha}
  \left(\frac{2\log_p q}{q^\alpha}+\frac{p^\alpha}{(p^\alpha-1)^2}\right)^{-1}\\
  &= \frac{1}{\alpha d(p,q)},\,\text{ where } d(p,q)=
  \frac{p^\alpha-2}{p^\alpha}\ln q+\frac{1}{p^\alpha-1}\ln p.
\end{align}

Notice that, by~\eqref{defalpha} again, $d(p,q)-d(q,p)=0$, so that $d(p,q)$ may
be replaced by the half-sum of $d(p,q)$ and $d(q,p)$, which yields the claimed
bound for $C_{p,q}$.  
\end{proof}



\providecommand{\bysame}{\leavevmode\hbox to3em{\hrulefill}\thinspace}
\providecommand{\MR}{\relax\ifhmode\unskip\space\fi MR }
\providecommand{\MRhref}[2]{%
  \href{http://www.ams.org/mathscinet-getitem?mr=#1}{#2}
}
\providecommand{\href}[2]{#2}

\end{document}